\documentclass{amsart}

\usepackage{amsmath, amssymb, amscd}
\usepackage{color}
\usepackage{graphicx}
\usepackage{enumitem}
\usepackage{comment}
\usepackage{hyperref}

\hypersetup{
    colorlinks=true,
    linkcolor=blue,
    filecolor=magenta,      
    urlcolor=pink,
    pdftitle={Laguerre-type Differential Operator},
    pdfpagemode=FullScreen,
}

\numberwithin{equation}{section}

\newtheorem{theorem}{Theorem}[section]

\newtheorem{corollary}[theorem]{Corollary}
\newtheorem{lemma}[theorem]{Lemma}
\newtheorem{proposition}[theorem]{Proposition}

\newtheorem{remark}[theorem]{Remark}

\renewcommand{\Re}{\operatorname{Re}}
\renewcommand{\Im}{\operatorname{Im}}

\providecommand{\keywords}[1]
{
  \small	
  \textbf{\textit{Keywords---}} #1
}

\title[$2\times 2$ Laguerre-type differential operator]{\boldmath $2\times 2$ Laguerre-type differential operator with triangular eigenvalue}

\author{Yanina González} 
\address[Yanina González]{Facultad de Ciencias Exactas y Naturales, Universidad Nacional de Cuyo, Mendoza, Argentina}
\email{ygonzalez@fcen.uncu.edu.ar}

\author{Victoria Torres}
\address[Victoria Torres]{CIEM-FaMAF, Universidad Nacional de Córdoba, Argentina}
\email{victoria.torres.999@unc.edu.ar}

\begin{document}

\begin{abstract}
In this paper, we present a comprehensive account of all Laguerre-type differential operators $D$ that are symmetric with respect to a $2\times 2$ irreducible weight $W$ on the interval $(0, \infty)$. These operators are associated with monic orthogonal polynomials ${P_n}$, which satisfy the equation $DP_n = P_n\Delta_n$ for a certain lower triangular eigenvalue $\Delta_n$. We introduce three distinct families of operators and weights, each characterized by explicit expressions depending on two or three parameters, along with a new expression based on a single parameter.
\end{abstract}



\thanks{This paper was partially supported by  CONICET, PIP
N°:11220200102031, SeCyT-UNC; and SIIP - Universidad Nacional de Cuyo
number $06/M032-T1$.}

\maketitle

\section{Introduction}\label{introduction}

    The foundational work on the general theory of matrix-valued orthogonal polynomials was initiated by M. G. Krein in 1949 \cite{Kr49}. These polynomials satisfy the orthogonality condition given by $\int P_n(t)^*W(t)P_m(t)dt=0$ if $n\neq m$, where $W$ is a matrix weight and each polynomial $P_n$ is a matrix polynomial of degree $n$ with a nonsingular leading coefficient.
    Subsequent research has systematically explored sequences of orthogonal polynomials $\{P_n\}_{n\in\mathbb{N}_0}$ that serve as eigenfunctions of a second-order differential operator $D$. This was initially considered in \cite{D97}, while the first examples were constructed some time after in \cite{GPT01}, \cite{grunbaum2002matrix}, \cite{grunbaum2003matrix}, and \cite{DG04}.
   
    The survey of families of orthogonal matrix polynomials that satisfy second-order differential operators reveals significant differences compared to the scalar case. In 1929, Bochner demonstrated that the only families of orthogonal scalar polynomials satisfying second-order differential equations are the classical Hermite, Laguerre, and Jacobi polynomials \cite{Boch29}. However, a classification analogous to Bochner's results in the matrix context has not yet been fully attained, despite the ground-breaking work \cite{casper2022matrix} in which a general classification is given when certain natural condition is met. This work makes use, among other things, of a rich algebra of differential operators that was previously studied in \cite{CG06}, highlighting the increased complexity of the matrix case.

    In this work, we are only interested in matrix-valued orthogonal polynomials being eigenfunctions of  Laguerre-type operators, specifically second-order differential operators of the form
    \begin{equation*}
    D=tI\partial^{2}+\left(  C-tU\right)  \partial-V  , \quad \partial=\frac{d}{dt},
    \end{equation*}
    with $C,V,U\in\mathbb{C}^{2\times2}$ and $I$ being the identity matrix of size $2$. In the scalar case, such operators have the Laguerre polynomials as their eigenfunctions. However, the corresponding problem for matrix-valued functions remains unresolved. For previous contributions to the study of Laguerre-type matrix polynomials, we refer readers to \cite{CMV07, CG06, D09b, DG04, DG07, DdI08, DL-R07, DS13}. It is worth noting that the methodologies employed in these earlier works primarily focus on the algebras of operators, characterization of matrix weights, and structural formulas associated with orthogonal polynomials, which differ significantly from the approaches used in this article.
    
    We aim to classify {\it explicetely} all $2\times 2$ Laguerre-type operators symmetric $D$ with respect to a $2\times 2$ irreducible matrix weight $W$ with 
    the assumption that the eigenvalues for the (unique) sequence of monic orthogonal polynomials $\{P_n\}_{n\in\mathbb{N}_0}$ are lower triangular matrices $\{\Delta_n\}_{n\in\mathbb{N}_0}$, that is to say $DP_n=P_n\Delta_n$ for all $n\geq 0$. Furthermore, the explicit matrix weights $W$ for which the differential operators are $W$-symmetric are also given.

    The main result is in Section \ref{sec7}, where we prove that any operator with these properties is equivalent to an operator in one of the following three families.
    
    The first family is given by
        \begin{align*}
        D_{\alpha,\beta,b}&=t I\partial^2 +\begin{pmatrix}
                \alpha+1+\beta-t & 0\\
                -b(\beta-2)t & \alpha+1-t
             \end{pmatrix}\partial -\begin{pmatrix}
                 1 & 0\\
                 -b(\alpha+1) & 0
             \end{pmatrix},\\
        W_{\alpha,\beta,b}(t)&=e^{-t}t^{\alpha} \begin{pmatrix}
            t^\beta+b^2t^2 & b t\\[1em]
            b t & 1
        \end{pmatrix},\quad \alpha>-1,\, \beta>-1-\alpha,\, b\neq 0.
    \end{align*}
    
    The second family is given by
        \begin{align*}
        D_{\alpha,b}&=tI\partial^2+  \begin{pmatrix}
                \alpha+5-t & 0\\
                -4b(\alpha+2)t & \alpha+1-t
             \end{pmatrix}\partial -\begin{pmatrix}
                 2 & 0\\
                 -2b(\alpha+2)(\alpha+1) & 0
             \end{pmatrix},\\
             W_{\alpha,b}(t)&=e^{-t}t^{\alpha} 
        \begin{pmatrix}
            t^4+4b^2(\alpha+2)t^2(\alpha+2-t)
            & -b t(t-2(\alpha+2))\\[1em]
            -b t(t-2(\alpha+2))
            & 1
        \end{pmatrix}, \quad \alpha>-1,\, 0<|b|<1.
    \end{align*}
    
    The third family is given by
    \begin{align*}
        D_{\beta}&=tI\partial^2+  \begin{pmatrix}
                3/2-t & \beta/4\\
                t & 1/2-t
             \end{pmatrix}\partial -\begin{pmatrix}
                 1/2 & 0\\
                 -1/2 & 0
             \end{pmatrix},\\
             W_{\beta}(t)&=e^{-t}t^{-1/2} \begin{pmatrix}
                \frac{4t}{\beta}(e^{\sqrt{\beta t}}+e^{-\sqrt{\beta t}})
                &  \frac{2\sqrt{t}}{\sqrt{\beta}}(e^{\sqrt{\beta t}}-e^{-\sqrt{\beta t}}) \\[1em]
                \frac{2\sqrt{t}}{\sqrt{\beta}}(e^{\sqrt{\beta t}}-e^{-\sqrt{\beta t}})
                & e^{\sqrt{\beta t}}+e^{-\sqrt{\beta t}}
            \end{pmatrix}, \quad \beta>0.
    \end{align*}

    \bigskip
    
    To build this classification, we first find necessary and sufficient conditions that matrices $C$, $U$ and $V$ must satisfy. 
    After some preliminaries given in Section \ref{preliminaries}, in Section \ref{sec3} we deduce some properties that the operator $D$ and its associated sequence of monic orthogonal polynomials. More precisely, we obtain the expression of the eigenvalue of $D$ in terms of $U$ and $V$ (Lemma \ref{Deltan}), the recurrence formulas for the coefficients of $P_n$ (Lemma \ref{FnGn}), and the coefficients of the three-term recurrence relation that $\{P_n\}_{n\in\mathbb{N}_0}$ satisfies (Lemma \ref{RR3T}).

    After Section \ref{sec3}, we consider three different cases according to the possible Jordan canonical form of the matrix $U$. We study separately each case in Sections \ref{sec4}, \ref{sec5}, and \ref{sec6}. In each one of these cases, we use the results of Section \ref{sec3} for determining the conditions under which the eigenfunctions $\{P_n\}_{n\in\mathbb{N}_0}$ of the operator $D$ is a sequence of monic orthogonal polynomials with respect to some weight $W$. After that, we give explicitly the expression of the weight $W$ for which $D$ is $W$-symmetric. 
    
    The first two families were previously discussed in  \cite{D09b} (see Theorem 4.2 and Theorem A.3, respectively). Previously, specific instances of the first family have been examined. For $\beta=0$ see \cite[p.16]{CG06}, for $\beta=1/2$ see \cite[p.46]{DG07} and for $\beta=1$ see \cite[p.170]{DdI08}. To the best of our knowledge, the third family was never considered before and it is completely new.
    
\section{Preliminaries}\label{preliminaries}
    
    A complex matrix-valued function $W(t)$ of size $N$ is a matrix weight if it is an integrable function on the possibly infinite interval $\left(  a, b\right)$ such that $W(t)$ is positive definite almost everywhere and with finite moments of all orders. 
    
    We shall denote by $\mathbb{C}^{N\times N}[t]$ the $\mathbb{C}^{N\times N}$-module of all polynomials in the indetermined $t$ with complex-matrix coefficient of size $N$. We assume the matrix weight $W$ to be non-degenerate, and define the matrix inner product over the linear space $\mathbb{C}^{N\times N}[t]$ as
    \[
    \left\langle P,Q\right\rangle _{W}=\int_{a}^b P\left(  t\right)^*
    W\left(  t\right)  Q\left(  t\right)  dt,
    \]
    where $P\left(  t\right)^*$
    denotes the conjugate transpose of $P\left(  t\right)$.
    By the Gram-Schmidt process, a unique sequence of matrix polynomials $\{P_n\}_{n\in\mathbb{N}_0}$ is established to be orthogonal with respect to $W$, satisfying  $\deg(P_n)=n$ and being monic for all $n\geq 0$. Furthermore, any sequence $\{Q_n\}_{n\in\mathbb{N}_0}$ of matrix polynomials orthogonal with respect to $W$ follows the structure $Q_n(t)=P_n(t)M_n$ where $\{M_n\}_{n\in\mathbb{N}_0}$ is some sequence of non-singular matrices.
    
    If $\{Q_n\}_{n\in\mathbb{N}_0}$ constitutes a sequence of orthonormal matrix polynomials with respect to a weight matrix $W$, then $\{Q_n\}_{n\in\mathbb{N}_0}$ satisfies a three-term recurrence relation
    \begin{equation*}
    tQ_{n}\left(  t\right)  =Q_{n+1}\left(  t\right)A_{n+1}  +Q_{n}\left(  t\right)
    B_{n}+Q_{n-1}\left(  t\right)A^{*}_{n}, \quad n\geq 0,
    \end{equation*}
    where $A_{n}$ is nonsingular and $B_{n}$ is Hermitian, with the convention $Q_{-1}(t)=0$. Conversely, by Favard's Theorem \cite{DLR96}, this three-term recurrence relation for matrix polynomials characterizes the orthogonality of a sequence of matrix polynomials. By using this result, one can deduce that  the sequence of monic matrix polynomials $\{P_n\}_{n\in\mathbb{N}_0}$ verifies a three-term recurrence relation 
    \begin{equation}\label{MonicRR}
    tP_{n}\left(  t\right)  =P_{n+1}\left(  t\right)  +P_{n}\left(  t\right)
    B_{n}+P_{n-1}\left(  t\right)A_{n}, \quad n\geq 0,
    \end{equation}
    with $P_{-1}(t)=0$. The existence of a weight matrix $W$ relies on two conditions on the coefficients of the relation: $A_{n}=S_{n-1}^{-1}S_{n}$ is a nonsingular matrix for all $n\geq 1$ and $S_{n}B_{n}$ is Hermitian for all $n\geq 0$, where $S_n=\langle P_n, P_n\rangle_W$.
    
    \bigskip

    A sequence of orthogonal polynomials       $\{P_n\}_{n\in\mathbb{N}_0}$ is eigenfunction of a left-hand side second-order differential operator $D=F_2(t)\partial^2+F_1(t)\partial+F_0$ if verifies the equation  
    \begin{equation}\label{EDO}
        F_2(t)P_n''(t)+F_1(t)P_n'(t)+F_0P_n(t)=P_n\Delta_n, \quad n\geq 0,
    \end{equation}
    where $\{\Delta_n\}_{n\in\mathbb{N}_0}$ are certain matrices. We employ the abbreviated notation $DP_n=P_n\Delta_n$. Specifically, in \cite[Proposition 2.10]{GT} authors proved that it is sufficient to examine differential operators $D$ such that $\deg(F_i)\leq i$ and are $W$-symmetric. A differential operator $D$ is said to be \emph{$W$-symmetric} if it satisfies the symmetry condition, defined as follows: $\left\langle DP, Q\right\rangle _{W}=\left\langle P, DQ \right\rangle _{W}$ for all matrix-valued polynomials $P$ and $Q$. 
    
    The symmetry condition is equivalent to a set of differential equations, known as \emph{symmetry equations}, which are presented in the following theorem.
    
    \begin{theorem}\label{symcond}
    \cite[Theorem 3.1]{DG04} 
    Let $D$ be a second-order differential operator of the form 
        \[
        D=F_2(t)\partial^2+F_1(t)\partial+F_0
        \]
        and let $W$ be a matrix weight supported on the interval $(a,b)$. Then $D$ is $W$-symmetric if and only if the following three differential equations are satisfied
        \begin{gather*}
                F_2^\ast W  = W F_2 ,\\
                2 (W F_2)'  = W F_1 + F_1 ^\ast W ,\\
                (W F_2)'' = (W F_1)' - W F_0 +F_0^* W,
        \end{gather*}
        with the boundary conditions  $\lim_{t \to l} W F_2  = 0$ and $\lim_{t \to l} W F_1 - F_1^\ast W = 0$ for $l = a, b$.
    \end{theorem}

    \bigskip
    
    In this context, we establish the following relation: let $D$ be $W$-symmetric and $\Tilde{D}$ be $\Tilde{W}$-symmetric, then the pair $(W, D)$ is said to be \emph{equivalent} to the pair $(\Tilde{W},\Tilde{D})$ if there is a nonsingular matrix $M$ such that $\Tilde{D}=M^{-1}D M$ and $\Tilde{W}=M^* W M$. Moreover, if $\{P_n\}_{n\in\mathbb{N}_0}$ is a sequence of orthogonal polynomials with respect to $W$, then $\{M^{-1}P_nM\}_{n\in\mathbb{N}_0}$ is a sequence of orthogonal polynomials with respect to $\Tilde{W}$. 

   In particular, if $\Tilde{W}$ is a diagonal matrix with scalar weights in its entries, we say that $W$ \emph{reduces to a scalar weight}. A necessary condition for reducibility is provided in \cite[Theorem 4.3]{TZ} and can be reformulated as follows:
        \begin{theorem}\label{opzeroorder}
        If a matrix weight $W$ supported on $(a,b)$ reduces, then there is a non-scalar matrix $V_0$ such that $W(t)V_0=V_0^*W(t)$.
        \end{theorem}
        
        Likewise, it can be stated that a family of orthogonal polynomials $\{P_n\}_{n\in\mathbb{N}_0}$ \emph{reduces} if there is an invertible matrix $M$ such that $\{M^{-1}P_nM\}_{n\in\mathbb{N}_0}$ is a diagonal matrix with classic polynomials in its entries. An equivalent condition can be found in the following result.
        \begin{theorem}\label{Pnreducible}\cite[Theorem 4.9]{TZ}
        A sequence of monic orthogonal polynomials $\{P_n\}_{n\in\mathbb{N}_0}$ with respect to a weight $W$ is reducible if and only if its commutant, given by
        \[
        \{T\in\operatorname{Mat}_n(\mathbb{C})\mid TP_n(t)=P_n(t)T, \, \text{for all }n,t\},
        \]
        contains non-scalar matrices.
    \end{theorem}

\section{First properties for the operator \texorpdfstring{$D$}{TEXT}}\label{sec3}

This work presents a classification of second-order differential operators
        \begin{equation*}
            D = t I\partial^2 + (C - t U ) \partial - V, \quad t\in(0,\infty),
        \end{equation*}
    where $C,U,V\in \mathbb{C}^{2\times 2}$, that are symmetric with respect to a $2\times 2$ matrix weight $W$, with irreducible monic orthogonal polynomials $\{P_n\}_{n\in\mathbb{N}_0}$. Specifically, we require that $DP_n=P_n\Delta_n$ for all $n\geq 0$, with eigenvalues $\{\Delta_n\}_{n\in\mathbb{N}_0}$ assumed to be complex lower triangular matrices.

    \bigskip
    
    \begin{proposition}\label{UVtriang}
        The matrices $U$ and $V$ are lower triangular and 
        \begin{equation} \label{Deltan}
            \Delta_n = -n U - V, \quad n\geq 0.
        \end{equation}
    \end{proposition}
    
    \begin{proof}
        Assume $P_{n}(t) = \sum_{k=0}^n T_k^n t^k$ with $T_n^n = I$. Now, replace it in the differential equation $DP_{n} = P_{n}\Delta_{n}$ and compare the leading coefficients on both sides of the equality to obtain \eqref{Deltan}.
    \end{proof}
    
    \bigskip

    It can be observed that the sequence of polynomials $P_n$ simultaneously satisfies the equations $DP_n=P_n\Delta_n$ and $(D-v_{22}I)P_n=P_n(\Delta_n-v_{22}I)$. Therefore, we take $v=v_{11}$ and $v_{22}=0$. This yields the explicit matrix form given by
    \begin{equation} \label{problema}
    D=\begin{pmatrix}
        t & 0\\
        0 & t
    \end{pmatrix}\partial^2 +\begin{pmatrix}
        c_{11}-tu_{11} & c_{12}\\
        c_{21}-tu_{21} & c_{22}-tu_{22}
            \end{pmatrix}\partial-\begin{pmatrix}
                    v & 0\\
                    v_{21} & 0
                \end{pmatrix}.
    \end{equation}
    
    \bigskip
    
    \begin{lemma} \label{Sn}
         Let $S_n=\langle P_n,P_n\rangle = \begin{pmatrix}
            s_{11}^n & s_{12}^n\\
            \overline{s_{12}^n} & s_{22}^n
        \end{pmatrix}$ be the square norm of the $n$-th polynomial $P_n$ and denote
        \begin{equation}\label{lambamunu}
        \Delta_{n}=\begin{pmatrix}
                \lambda_n & 0\\
                \nu_n & \mu_n
            \end{pmatrix}=
            \begin{pmatrix}
                -n\,u_{11}-v & 0\\
                -n\, u_{21}-v_{21}& -n\, u_{22}
            \end{pmatrix}.
        \end{equation} 
        The entries $u_{11},u_{22},v$ in \eqref{problema} are real and the following relations hold for all $n\geq 0$
        \begin{equation} 
        \begin{split}\label{nun}
             s^n_{22} \nu_n &= ( \mu_n - \lambda_n ) \overline{ s^n_{12} }\\
            - s^n_{22} (nu_{21}+v_{21})&=(-nu_{22}+nu_{11}+v) \overline{ s^n_{12} }.
        \end{split}
        \end{equation} 
        In particular, 
        \begin{equation} \label{v21}
            v_{21} = -v \frac{  \overline{ s^0_{12} } } { s^0_{22} }.
        \end{equation}
    \end{lemma}
    
    \begin{proof}
        It is evident that $S_{n}$ is a symmetric positive definite matrix.  Since $D$ is $W$-symmetric,  we have that $\langle DP_n,P_n\rangle=\langle P_n,DP_n\rangle$ and $DP_n=P_n\Delta_n$. Then, it follows that $ \Delta_n^{\ast} S_n = S_n \Delta_n$, i.e.
        \[
        \begin{pmatrix}
        \overline{\lambda_n}s_{11}^n +\overline{s_{12}^n\nu_n}
        & \overline{\lambda_n}s_{12}^n+\overline{\nu_n}s_{22}^n\\
        \overline{\mu_ns_{12}^n} 
        & \overline{\mu_n}s_{22}^n
        \end{pmatrix}=\begin{pmatrix}
            s_{11}^n\lambda_n+s_{12}^n\nu_n
            &s_{12}^n\mu_n\\
            \overline{s_{12}^n}\lambda_n+s_{22}^n\nu_n
            &s_{22}^n\mu_n
        \end{pmatrix}.
        \]
        
        From entry $(2,2)$, we get that that $\mu_{n} \in \mathbb{R}$ for all $n\in\mathbb{N}_{0}$, and from entry $(2,1)$ we obtain $\nu_n=\frac{\overline{s^n_{12}}}{s_{22}^n}(\mu_n-\lambda_n).$
        
        Substitute this expression in entry $(1,1)$ to obtain $ 0=(\lambda_n-\overline{\lambda_n}) \left( s^n_{11}
        -\frac{|s^n_{12}|^2}{s^n_{22}}\right)$.
        As $S_n$ is positive definite, if follows that $\lambda_n\in \mathbb{R}$ for all $n\geq 0$. 

        Now, from \eqref{lambamunu}, we have  $u_{11},u_{22},v\in\mathbb{R}$.
    \end{proof}

    \bigskip
    
    In the rest of the work we denote the $k$-th coefficient of the monic polynomial $P_n(t)$ by $T_k^n$, and its columns by $F_k^n$ and $G_k^n$. We use these expressions to find the coefficients of the recurrence relation for $P_n(t)$ that will be a relevant part of the procedure in the sections to come.

    \begin{lemma} \label{FnGn} 
       Let $\{P_n(t)\}_{n\in\mathbb{N}_0}$ be the sequence of monic orthogonal polynomials satisfying $DP_n=P_n\Delta_n$, where $D$ is given in \eqref{problema}. Then, 
        the coefficients of $P_n(t)$ are defined recursively by
        \begin{equation}\label{Fn}
            \begin{split}
                F^n_n & =   \begin{pmatrix}
                            1\\ 0
                            \end{pmatrix}, \\
                \left(  \lambda_{n} I - \Delta_{k} \right)  F_{k}^{n} & = 
                    \left(  k + 1 \right)  \left( C + k I \right)  F_{k+1}^{n} - \nu_{n} G_{k}^{n}, \quad 0 \leq k \leq n-1;
            \end{split} 
        \end{equation}
        and
        \begin{equation} \label{Gn}
            \begin{split}
                G_n^n & =   \begin{pmatrix}
                            0\\1
                            \end{pmatrix},\\
            \left(  \mu_{n} I - \Delta_{k} \right)  G_{k}^{n} & =\left(  k + 1 \right)  \left(
            C + k I \right)  G_{k+1}^{n}, \quad 0 \leq k \leq n-1. 
            \end{split}    
        \end{equation}
    \end{lemma}
    
    \begin{proof}
        Substitute $P_n(t)$ in \eqref{problema} and use the fact that $D t^k I = k (k-1) t^{k-1} I + k (C - tU) t^{k-1} - V t^k$ for all $0\leq k\leq n$.
    \end{proof}
    
    \begin{lemma} \label{AnBn}
        Given the sequence of monic polynomials $\{P_n(t)\}_{n\in\mathbb{N}_0}$, the coefficients of the three-term recurrence relation for $P_n(t)$
        \begin{equation}\label{RR3T}
            tP_n(t)=P_{n+1}(t)+P_n(t)B_n+P_{n-1}(t)A_n,
        \end{equation}
        can be expressed as
        \begin{align*}
            B_{n}  &  = T_{n-1}^{n} - T_{n}^{n+1},\quad n \geq 1, \quad B_{0}= - T_{0}^{1}\\
            A_{n}  &  = T_{n-2}^{n} - T_{n-1}^{n+1} - T_{n-1}^{n} B_{n},\quad n \geq 1.
        \end{align*} 
    \end{lemma}
    
    \begin{proof}
        The proof proceeds by substituting $P_n(t)$ into \eqref{RR3T} and comparing the coefficients of $t^{n+1}$, $t^n$, and $t^{n-1}$ on both sides of the equality.
    \end{proof}
    
    \bigskip

    Based on the findings presented in this section,  we outline a systematic procedure for the identification of $W$-symmetric differential operators $D$. Firstly, the $(n-1)$-th and $(n-2)$-th coefficients of $P_n$ are calculated, which are represented by
    \[
        T_{n-1}^n=
        \begin{pmatrix}
            f_1^{n} & g_1^n\\
            f_2^{n} & g_2^{n}
        \end{pmatrix} \quad \text{and} \quad 
        T_{n-2}^n=
        \begin{pmatrix}
            f_3^{n} & g_3^{n}\\
            f_4^{n} & g_4^{n}
        \end{pmatrix}.
    \]
    In this manner, the coefficients $A_n$ and $B_n$ of the recurrence relation for $P_n$ are determined. Subsequently, the conditions stated in Favard's theorem, namely $S_n=S_{n-1}A_n$ and $S_nB_n$ Hermitian (where $S_n=\langle P_n, P_n\rangle$), are applied to derive constraints on the parameters characterizing $D$. Finally, the weight W is calculated using the symmetry equations detailed in Theorem \ref{symcond}.
    The determination of the coefficients $T_{n-1}^n$ and $T_{n-2}^n$ depends on the invertibility of the matrices $\lambda_nI-\Delta_k$ and $\mu_nI-\Delta_k$ for $k=1,2$. Consequently, we divide the analysis of $D$ into three cases. Due to the operator $D = t I\partial^2 + (C - t U ) \partial - V$ is equivalent to 
        \[
        M^{-1}DM=t I \partial^2 + (M^{-1}CM - t M^{-1}UM ) \partial - M^{-1}VM,
        \]
        for any invertible matrix $M$, we consider separately the cases delineated by the Jordan canonical forms of the matrix $U$:
    
    \[
        U= \begin{pmatrix} 
                    u & 0\\
                    1 & u
                    \end{pmatrix},
            \quad U=\begin{pmatrix} 
                    u & 0\\
                    0 & u
                    \end{pmatrix}\quad
            \text{and}
            \quad
            U=\begin{pmatrix}
                    u_{1} & 0\\
                    0 & u_{2}
                    \end{pmatrix},\quad (u_1\neq u_2).
    \]    
    
\bigskip    

\section{\texorpdfstring{$U$}{TEXT} is a non-diagonal matrix}\label{sec4}    

    This section focuses on the case where the Jordan canonical form of the matrix $U$ is non-diagonal. Specifically, we characterize all $W$-symmetric operators $D$ of the form
    \begin{equation}\label{Dtriang}
        D=tI\partial^2   +   \begin{pmatrix}
                    c_{11} - t u & c_{12}\\
                    c_{21} - t  & c_{22} - t u
                \end{pmatrix} \partial
                -   \begin{pmatrix}
                    v & 0 \\
                    v_{21} & 0
                \end{pmatrix},
    \end{equation}
    where $u,v\in\mathbb{R}$, $v_{21}, c_{ij}\in\mathbb{C}$ and $t\in(0,\infty)$ such that the family of monic orthogonal polynomials $\{P_n(t)\}_{n\in\mathbb{N}_0}$ verifies $DP_n=P_n\Delta_n$ for all $n\geq 0$ with
    \[
    \Delta_n=\begin{pmatrix}
        \lambda_n & 0\\
        \nu_n & \mu_n
    \end{pmatrix}
    =\begin{pmatrix}
        -nu-v & 0\\
        -n-v_{21} & -nu
    \end{pmatrix}.
    \]

    \begin{proposition}\label{v21eq0}
        If the second-order differential operator \eqref{Dtriang} satisfies $DP_n=P_n\Delta_n$, then $v_{21}=0$ and $u,v\neq 0$.
    \end{proposition}

    \begin{proof}
        We may consider $v_{21}=0$ without losing generality in operator $D$. If $v_{21}$ were to have a non-zero value, equation \eqref{v21} would require a non-zero $v$, thereby making the operator \eqref{Dtriang} equivalent to one with $v_{21}=0$, following conjugation by
        $M=\begin{pmatrix}
                    \frac{v}{v_{21}} & 0 \\
                    1 & \frac{v}{v_{21}}
                \end{pmatrix}$.
        
        Now suppose $u=0$. Then equations \eqref{Fn} and \eqref{Gn} for $k=n-1$ become
        \[
            \begin{pmatrix}
            0\\
            (n-1)f_1^{n}-vf_2^{n}
        \end{pmatrix} 
        =
        n\begin{pmatrix}
            c_{11}+n-1+g_1^n\\
            c_{21}+g_2^{n}
        \end{pmatrix}\quad\text{and}\quad
            \begin{pmatrix}
            vg_1^n\\
            (n-1)g_1^n
            \end{pmatrix}=
            n\begin{pmatrix}
            c_{12}\\
            c_{22}+n-1
            \end{pmatrix}.
        \]
        However, it is not possible for the first entry in the left equality and the second entry in the right equality to hold simultaneously for all $n\geq 0$.
    \end{proof}

        \begin{theorem}\label{Teovequ}
            Let $D$ be the differential operator \eqref{Dtriang} with $|v|=|u|$. Then, $D$ is $W$-symmetric for some weight $W$ if, and only if is of the form
            \[
             D=tI\partial^2+\begin{pmatrix}
                   c_{22}+2+\dfrac{2c_{22}}{uc_{21}-c_{22}} -tu & 0\\
                   c_{21}-t & c_{22}-tu
                \end{pmatrix}\partial-\begin{pmatrix}
                    u & 0\\ 0 & 0
                \end{pmatrix},
            \]
            with $u>0$, $c_{22}>0$, $c_{22}>\frac{2c_{21}u}{c_{22}-c_{21}u}$ and $c_{21},c_{22}\in\mathbb{R}$. Moreover, $D$ is symmetric with respect to the irreducible weight
            \[
             W(t)=e^{-tu}\,t^{c_{22}-1}\begin{pmatrix}
                \gamma \, t^{\frac{2c_{21}u}{c_{21}u-c_{22}}}+\left(\dfrac{(c_{22}-c_{21}u)(c_{22}-tu)}{2uc_{22}}\right)^2
                & \dfrac{(c_{22} - c_{21} u) (c_{22} - t u)}{2u c_{22}}\\[1em]
                \dfrac{(c_{22} - c_{21} u) (c_{22} - t u)}{2u c_{22}}
                & 1
            \end{pmatrix}, \quad \gamma>0.
            \]
        \end{theorem}
          
        \begin{proof} 
        Equations \eqref{Fn} and \eqref{Gn} for $k=n-1$ become
        \begin{equation*}
        T_{n-1}^n=\begin{pmatrix}
            -\frac{n\bigl(c_{11}+n-1+g_1^n\bigr)}{u}& g_1^n\\[.5em]
            f_2^{n} &  \frac{(n-1)g_1^n-n\bigl(c_{22}+n-1\bigr)}{u}
        \end{pmatrix}, \quad n\geq 0,
        \end{equation*}
        where non-diagonal entries verify the equations
        \begin{align}
            (v-u)g_1^n&=nc_{12}, \label{relg1}\\
            \quad (u+v)f_2^{n}&=-n\frac{(n-1)\bigl(c_{11}-1+2g_1^n\bigr)+c_{21}u-nc_{22}}{u}. \label{relf2}
        \end{align}
        Now we compute $B_n=T_{n-1}^n-T_{n}^{n+1}$ and use the fact that $S_nB_n$ is an Hermitian matrix for all $n\geq 0$, where $s_{12}^n=-\frac{ns_{22}^n}{v}$ by identity \eqref{nun}. It is convenient to split the calculations into the two possibilities for $v$. 
        
        First suppose $v=u$. Relation \eqref{relg1} implies $c_{12}=0$. Then
        \[
            B_n=\begin{pmatrix}
            \frac{c_{11}+2n-ng_1^n+(n+1)g_1^{n+1}}{u} & g_1^n-g_1^{n+1}\\[.5em]
            \frac{2n\bigl(c_{11}-c_{22}-1-(n-1)g_1^n+(n+1)g_1^{n+1}\bigr)-c_{22}+c_{21}u}{2u^2} & \frac{c_{22}+2n-ng_1^{n+1}+(n-1)g_1^n}{u}
        \end{pmatrix}, \quad n\geq 0.
        \]
        
        The entry $(2,2)$ of $S_nB_n$ is given by $\frac{s_{22}^n\bigl(c_{22}+2n-g_1^n\bigr)}{u}$. The Hermiticity of this $S_nB_n$ for all $n\geq 0$ implies that $c_{22}\in\mathbb{R}$ and $g_1^n\in\mathbb{R}$ for all $n\geq 1$. 
        
        The entry $(1,1)$ of $S_nB_n$ is
        \[
        s_{11}^n\frac{c_{11}+2n-ng_1^n+(n+1)g_1^{n+1}}{u} +ns_{22}^n\frac{2n\bigl(c_{22}-c_{11}+1+(n-1)g_1^n-(n+1)g_1^{n+1}\bigr)+c_{22}-c_{21}u}{2u^3}.
        \]
        In particular, for $n=0$ it implies 
        $c_{11}\in\mathbb{R}$. Hence, for $n\geq 1$ it follows that $c_{21}\in\mathbb{R}$ as well.

        \

        Now we solve the first and second-order symmetry equations 
        \[
        2(WF_2)'=WF_1+F_1^*W \quad \text{and} \quad (WF_2)''-(WF_1)'-F_0^*W+WF_0=0,
        \]
        where $F_2(t)=tI$, $F_1(t)=C-tU$, $F_0=-V$ and $W(t)=(w_{ij}(t))$ is a symmetric positive definite matrix. These equations become        
        \begin{align}
            0&=\begin{pmatrix}
            tw_{11}'+(tu-c_{11}+1)w_{11}+(t-c_{21})\Re(w_{12})    
            & tw_{12}'+\left(tu-\frac{c_{11}+c_{22}-1}{2}\right)w_{12}+\frac{t-c_{21}}{2}w_{22}\\[.5em]
            tw_{21}'+\left(tu-\frac{c_{11}+c_{22}-1}{2}\right)\overline{w_{12}}+\frac{t-c_{21}}{2}w_{22}
            &tw_{22}'+(tu-c_{22}+1)w_{22}
            \end{pmatrix},\label{edo1}\\
            0&=\begin{pmatrix}
            tw_{11}''+(2-c_{11}+tu )w_{11}'+uw_{11}-(c_{21}-t)w_{12}'+w_{12}
            & tw_{12}''+(2-c_{22}+tu)w_{12}'+2uw_{12}\\[.5em]
            t\overline{w_{12}}''+(2-c_{11}+tu)\overline{w_{12}}'-(c_{21}-t)w_{22}'+w_{22}
            & tw_{22}''+(2-c_{22}+tu)w_{22}'+uw_{22}
        \end{pmatrix}.\label{edo2}
        \end{align}
        
        From the entry $(2,2)$ of \eqref{edo1} we have
        \[
            w_{22}(t)=e^{-tu}t^{c_{22}-1}k_1,
        \]
        where $k_1$, $u$ and $c_{22}$ are positive constants for $w_{22}$ to be integrable and positive.

        If we solve the equation corresponding to the entry $(1,2)$ of \eqref{edo1} we have
        \[
        w_{12}(t)=e^{-tu}t^{\frac{c_{11}+c_{22}-2}{2}}\left(k_2+\frac{k_1}{2}h(t)\right),
        \]
        where $k_2\in\mathbb{C}$ and
        \[
        h(t)=\begin{cases}
            c_{21}\log t-t, & \text{if }c_{11}=c_{22},\\
            -c_{21}/t-\log t, & \text{if }c_{11}=c_{22}+2,\\
            \dfrac{2c_{21}}{c_{22}-c_{11}}t^{\frac{c_{22}-c_{11}}{2}}-\dfrac{2}{c_{22}-c_{11}+2}t^{\frac{c_{22}-c_{11}+2}{2}}, & \text{otherwise}.
        \end{cases}
        \]
        If we replace $w_{12}$ and $w_{22}$ in the entry $(1,2)$ of \eqref{edo2} we deduce that $k_2=0$ and $c_{11}=c_{22}+\frac{2uc_{21}}{uc_{21}-c_{22}}$. 
        
        Now, from the entry $(1,1)$ of \eqref{edo1} we have
        \[
        w_{11}(t)=e^{-tu}t^{c_{22}-1}\left(k_3 t^{\frac{2c_{21}u}{c_{21}u-c_{22}}}+k_1\left(\frac{(c_{22}-c_{21}u)(c_{22}-tu)}{2c_{22}u}\right)^2\right),
        \]
        where $k_3>0$ and $c_{22}>\frac{2c_{21}u}{c_{22}-c_{21}u}$ for $w_{11}$ to be integrable and positive.

        In addition, we can consider without loss of generality $k_1=1$, which gives the weight in the statement.

        Finally, we prove that this weight and its orthogonal monic polynomials $\{P_n\}_{n\in \mathbb{N_{0}}}$ are irreducible. Let $V_0=\begin{pmatrix}
            v_1 & v_2\\v_3&v_4
        \end{pmatrix}$ and suppose $W(t)V_0=V_0^*W(t)$. The entry $(2,2)$ of this equation is $(c_{22} - c_{21} u) (c_{22} - t u)\Im(v_2)=-2u c_{22}\Im(v_4)$,
        which implies $v_2,v_4\in\mathbb{R}$. Now, the entry $(2,1)$ is
        \[
        \frac{(c_{22} - c_{21} u) (c_{22} - t u)}{2u c_{22}}(v_1-v_4)=\left(\gamma t^{\frac{2u c_{21}}{c_{21} u-c_{22} }} + \frac{(c_{22} - c_{21} u)^2 (c_{22} - t u)^2}{4u^2 c_{22}^2 }\right)v_2-v_3.
        \]
        This implies that $v_2=0$, $v_1=v_4$ and $v_3=0$. As $V_0$ is a scalar matrix, by Theorem \ref{opzeroorder} we have that $W$ is irreducible. Now suppose the family $\{P_n\}_{n\in\mathbb{N}_0}$ reduces. Then there is an invertible matrix $M\in \mathbb{C}^{2 \times 2}$ such that 
            \[ 
            M^{-1}P_nM= \begin{pmatrix}
                p_n^1(t) & 0 \\
                0 & p_n^2(t)
            \end{pmatrix},
            \]
        where $p_n^i$ are classical orthogonal polynomials for $i=1,2$. It is straightforward to check that $P_n$ is eigenfunction of the zero-order differential operator $M\begin{pmatrix}
               1 & 0 \\
                0 & 0
            \end{pmatrix}M^{-1}$,
        because
            \[
            \begin{pmatrix}
               1 & 0 \\
                0 & 0
            \end{pmatrix}(M^{-1}P_nM)=(M^{-1}P_nM) \begin{pmatrix}
               1 & 0 \\
                0 & 0
            \end{pmatrix}, \quad n\geq 0.
            \]
        Since $W$ is irreducible, then $M\begin{pmatrix}
               1 & 0 \\
                0 & 0
            \end{pmatrix}M^{-1}$ is scalar. Which is not possible.
            
\

        Now suppose $v=-u$. Relation \eqref{relf2} implies $c_{12}=0$, $c_{22}=c_{11}-1$ and $c_{21}=\frac{c_{11}-1}{u}$. Then
        \[
        B_n=\begin{pmatrix}
            \frac{c_{11}+2n}{u} & 0\\
            f_2^{n}-f_2^{n+1} & \frac{c_{11}+2n-1}{u}
        \end{pmatrix}, \quad n\geq 0.
        \]
        The Hermiticity of 
        \[
        S_nB_n=\begin{pmatrix}
            \frac{s_{11}^n(c_{11}+2n)+ns_{22}^n\bigl(f_2^{n}-f_2^{n+1}\bigr)}{u}
            & &\frac{ns_{22}^n(c_{11}+2n-1)}{u^2}\\[1em]
            s_{22}^n\left(\frac{n(c_{11}+2n)}{u^2}+f_2^{n}-f_2^{n+1}\right) 
            & &\frac{s_{22}^n(c_{11}+2n-1)}{u}
        \end{pmatrix}, \quad n\geq 0,
        \]
        implies that $c_{11}\in\mathbb{R}$ and $f_2^{n+1}=f_2^{n}+\frac{n}{u^2}$.

        Now we use the relation $S_{n}=S_{n-1}A_n$ to prove there is no weight $W$ such that $D$ is $W$-symmetric. We compute the coefficient $T_{n-2}^n$ using Lemma \ref{FnGn} for $k=n-2$
        \[
            T_{n-2}^n=\begin{pmatrix}
                
                \frac{n(n-1)(c_{11}+n-2)(c_{11}+n-1)}{2u^2}
                & 0\\[.5em]
                -n(n-1)\frac{(n-3)(c_{11}+n-1)+2}{2u^3}
                &\frac{n(n-1)(c_{11}+n-3)(c_{11}+n-2)}{2u^2}
            \end{pmatrix}, \quad n\geq 1.
        \]
        Then by Lemma \ref{AnBn} we have
        \[
        A_n=T_{n-2}^n-T_{n-1}^{n+1}-T_{n-1}^nB_n=
        \begin{pmatrix}
            \frac{n(c_{11}+n-1)}{u^2} & 0\\[.5em]
            -\frac{n(4c_{11}+3n-7)}{2u^3} & \frac{n(c_{11}+n-2)}{u^2}
        \end{pmatrix},
        \quad n\geq 1.
        \]

        The entries $(1,2)$ and $(2,2)$ of $S_n=S_{n-1}A_n$ turn respectively into
        \[
        s_{22}^n=\frac{s_{22}^{n-1}(n-1)(c_{11}+n-2)}{u^2} \quad \text{and} \quad s_{22}^n=\frac{s_{22}^{n-1}n(c_{11}+n-2)}{u^2}, 
        \]
        which is a contradiction.
        \end{proof}

        \bigskip

        The remainder of this section is dedicated to an analysis of the case in which $|v|\neq |u|$. We follow a similar procedure to restrict the conditions on $D$ and find the weight $W$. For convenience, we divide subsequent analysis according to whether or not the operator $D$ is triangular.

        \begin{theorem}\label{Teoveq2u}
            Let $D$ be the differential operator \eqref{Dtriang} with $|v|\neq |u|$ and $c_{12}=0$. Then, $D$ is $W$-symmetric with respect to some weight $W$ if and only if is of the form
            \begin{equation*}
            D=tI \partial^2
            +   \begin{pmatrix}
                    c_{21}u+4 - u t  & 0\\
                    c_{21} - t  & c_{21}u - u t 
                \end{pmatrix} \partial
            -   \begin{pmatrix}
                    2u & 0 \\
                    0 & 0
                \end{pmatrix},
            \end{equation*}
            where $u>0$ and $c_{21}>0$. Moreover, $D$ is symmetric with respect to the irreducible weight function
             \[
            W(t)=e^{-tu} t^{c_{21}u-1} 
            \begin{pmatrix}
                t^4 \gamma + \dfrac{(c_{21}-2t)\bigl(c_{21}(c_{21}u+1)+2t(tu-1-c_{21}u)\bigr)}{16(c_{21}u+1)}
                & t \left(\dfrac{1}{2}-\dfrac{t u}{4(c_{21}u+1)}\right)-\dfrac{c_{21}}{4}\\[1em]
                t \left(\dfrac{1}{2}-\dfrac{t u}{4( c_{21}u+1)}\right)-\dfrac{c_{21}}{4}
                & 1
            \end{pmatrix},
            \]
           where $\gamma>\dfrac{u^2}{16(c_{21}u+1)^2}$.
        \end{theorem}
    
        \begin{proof}

        By Lemma \ref{FnGn} for $k=n-1$ and Lemma \ref{AnBn} we have
        \[
                T_{n-1}^n=\begin{pmatrix}
                \frac{-n(c_{11}+n-1)}{u}&0 \\[.5em]
                \frac{-n\bigl((n-1)(c_{11}-1)-nc_{22}+c_{21}u\bigr)}{u(u+v)}&  \frac{-n(c_{22}+n-1)}{u}
            \end{pmatrix}, \quad
            B_n=\begin{pmatrix}
                    \frac{c_{11}+2n}{u}
                    &0\\[.5em]
                    \frac{2n(c_{11}-1)-c_{22}(2n+1)+c_{21}u}{u(u+v)}
                    &\frac{c_{22}+2n}{u}
                \end{pmatrix},
        \]
        for all $n\geq 0$. Thus
        \[
            S_nB_n=\begin{pmatrix}
                \frac{s_{11}^n(c_{11}+2n)}{u}-ns_{22}^n\frac{2n(c_{11}-1)-c_{22}(2n+1)+c_{21}u}{uv(u+v)}
                &&-\frac{ns_{22}^n (c_{22}+2n)}{uv}\\[1em]
                \frac{s_{22}^n}{u}\left(\frac{2n(c_{11}-1)-c_{22}(2n+1)+c_{21}u}{u+v}-\frac{n(c_{11}+2n)}{v}\right)
                && \frac{s_{22}^n(c_{22}+2n)}{u}
            \end{pmatrix}, \quad n\geq 0.
        \]
        The Hermiticity of this matrix implies that $c_{11},c_{22}\in\mathbb{R}$, $c_{22}=c_{21}u$ and $c_{11}=c_{21} u - \frac{2v}{u-v}$.

        Now we solve the first and second-order symmetry equations 
        \[
        2(WF_2)'=WF_1+F_1^*W \quad \text{and} \quad (WF_2)''-(WF_1)'-F_0^*W+WF_0=0,
        \]
        where $F_2(t)=tI$, $F_1(t)=C-tU$, $F_0=-V$ and $W(t)=(w_{ij}(t))$ is a symmetric positive definite matrix. These equations become
        \begin{align}
                0&=\begin{pmatrix}
                tw_{11}'+\frac{(t-c_{21})(u-v)u+u+v}{u-v}w_{11}+(t-c_{21})\Re(w_{12})    
                & tw_{12}'+u\frac{(t-c_{21})(u-v)+1}{u-v}w_{12}+(t-c_{21})w_{22}\\[.5em]
                t\overline{w_{12}'}+u\frac{(t-c_{21})(u-v)+1}{u-v}\overline{w_{12}}+(t-c_{21})w_{22}
                &tw_{22}'+\bigl((t-c_{21})u+1\bigr)w_{22}
                \end{pmatrix},\label{edo1v2u}\\
                0&=\begin{pmatrix}
                tw_{11}''+u\frac{(t-c_{21})(u-v)+2}{u-v}w_{11}'+uw_{11}+\bigl((t-c_{21})w_{12}\bigr)'
                & tw_{12}''+\Bigl(\bigl((t-c_{21})u+2\bigr)w_{12}\Bigr)'\\[.5em]
                t\overline{w_{12}}''+u\frac{(t-c_{21})(u-v)+2}{u-v}\overline{w_{12}}'+\overline{w_{12}}(u-v)+\bigl((t-c_{21})w_{22}\bigr)'
                & tw_{22}''+\Bigl(\bigl((t-c_{21})u+2\bigr)w_{22}\Bigr)'
                \end{pmatrix}.\label{edo2v2u}
        \end{align}
        
        From entry $(2,2)$ of \eqref{edo1v2u} we find that 
        \[
        w_{22}(t)=e^{-tu}t^{c_{21}u-1}k_1,
        \]
        where $k_1$, $u$ and $c_{21}$ are positive constants for $w_{11}$ to be integrable and positive. This function verifies the entry $(2,2)$ of \eqref{edo2v2u}.

        If we solve the equation corresponding to entry $(1,2)$ of \eqref{edo1v2u} we have
        \[
        w_{12}(t)=e^{-tu}t^{u\left(c_{21}-\frac{1}{u-v}\right)}\left(k_2+\frac{k_1(u-v)}{2uv}\left(c_{21}ut^{\frac{v}{u-v}}-vt^{\frac{u}{u-v}}\right)\right),
        \]
        where $k_2\in\mathbb{C}$. Since $\int W(t)dt=S_0$ is a diagonal matrix, we have that 
        \[
        0=\int_{0}^{\infty} w_{12}(t)dt= \dfrac{c_{21}k_1(u-v)^2\Gamma(c_{21}u)}{2vu^{c_{21}u+1}}+k_2u^{\frac{v}{u-v}-c_{21}u}\Gamma\left(c_{21}u-\frac{v}{u-v}\right),
        \]
        where $c_{21}u>\frac{v}{u-v}$. Then
        \[
        k_2=-\frac{c_{21}k_1(u-v)^2u^{\frac{u}{v-u}}\Gamma(c_{21}u)}{2v\Gamma\left(c_{21}u-\frac{v}{u-v}\right)}.
        \]
        
        If we substitute $w_{12}$ and $w_{22}$ in the entry $(2,1)$ of \eqref{edo2v2u} we have 
        \[
        0=\frac{e^{-tu}t^{c_{21}u}k_1}{2u}\left(\frac{t^{\frac{2u-v}{v-u}}\Gamma(c_{21}u+1)\bigl(u+c_{21}u(v-u)+t(2u-v)(u-v)\bigr)}{u^{\frac{u}{u-v}}\Gamma\left(c_{21}u-\frac{v}{u-v}\right)}-(u-v)^2\right),
        \]
        which implies $v=2u$.
        
        Finally, from entry $(1,1)$ of \eqref{edo1v2u} we get
        \[
        w_{11}(t)=e^{-tu}t^{c_{21}u-1}\left(t^4k_3+\frac{k_1(c_{21}-2t)\bigl(c_{21}(c_{21}u+1)+2t(tu-1-c_{21}u)\bigr)}{16(c_{21}u+1)}\right),
        \]
        where $k_3\in\mathbb{C}$. Note that there is no loss of generality by taking $k_1=1$. Therefore, we have the integrable weight of the statement.
        
        It can be proved that this weight and the corresponding family of monic polynomials do not reduce, as we did in Theorem \ref{Teovequ}.     
        \end{proof}    

        \begin{theorem}\label{Teoueq2v}
            Let $D$ be the differential operator \eqref{Dtriang} with $|v|\neq |u|$ and $c_{12}\neq 0$. Then, $D$ is $W$-symmetric with respect to some weight $W$ if, and only if is
            \begin{equation*}
            D=tI \partial^2
            +   \begin{pmatrix}
                    2-c_{22} - tu  & \dfrac{u(1-2c_{22})}{2}\\[1em]
                    \dfrac{2c_{22}-3}{2u} - t  & c_{22} - tu 
                \end{pmatrix} \partial
            -   \begin{pmatrix}
                    \dfrac{u}{2} & 0 \\[1em]
                    0 & 0
                \end{pmatrix},
            \end{equation*}
            where $u>0$ and $c_{22}>1/2$. Moreover, $D$ is symmetric with respect to the irreducible weight
            \[
            W(t)=\frac{e^{-tu}}{\sqrt{t}}\begin{pmatrix}
                \dfrac{(\gamma^2+2tu) \cosh\bigl(\gamma\sqrt{2tu}\bigr)-2\gamma\sqrt{2tu}\sinh\bigl(\gamma\sqrt{2tu}\bigr)}{\gamma^2u^2}
                & \dfrac{\gamma\cosh\bigl(\gamma\sqrt{2tu}\bigr)-\sqrt{2tu}\sinh\bigl(\gamma\sqrt{2tu}\bigr)}{\gamma u}\\
                \dfrac{\gamma\cosh\bigl(\gamma\sqrt{2tu}\bigr)-\sqrt{2tu}\sinh\bigl(\gamma\sqrt{2tu}\bigr)}{\gamma u}
                & \cosh\bigl(\gamma\sqrt{2tu}\bigr)
            \end{pmatrix},
            \]
            where $\gamma=\sqrt{2c_{22}-1}$.
        \end{theorem}
        
    \begin{proof}
        By Lemma \ref{FnGn} with $k=n-1$ we obtain
            \begin{equation*}
                T_{n-1}^n=\begin{pmatrix}
                n\left(\frac{c_{11}-n+1}{u}+\frac{nc_{12}}{u(u-v)}\right)
                &\frac{nc_{12}}{v-u} \\[.5em]
                n\left(\frac{2n(n-1)c_{12}}{u(u^2-v^2)} -\frac{n(c_{11}-c_{22}-1)-c_{11}+c_{21}u+1}{u(u+v)}\right)
                &  -n\left(\frac{c_{22}+n-1}{u}+\frac{(n-1)c_{12}}{u(u-v)}\right)
            \end{pmatrix}, \quad n\geq 0.
            \end{equation*}
        Thus, by Lemma \eqref{AnBn} we have
            \begin{equation*}
                B_n=\begin{pmatrix}
                    \frac{c_{11}+2n}{u}-\frac{(2n+1)c_{12}}{u(u-v)}
                    &\frac{c_{12}}{u-v}\\[.5em]
                    -\frac{2n(c_{22}-c_{11}+1)+c_{22}-c_{21}u}{u(u+v)}-\frac{2n(3n+1)c_{12}}{u(u^2-v^2)}
                    & \frac{c_{22}+2n}{u}+\frac{2nc_{12}}{u(u-v)}
                \end{pmatrix}, \quad n\geq 0.
            \end{equation*}
        The entries of $S_nB_n$ are
        \begin{align*}
            (1,1): & \quad\dfrac{s_{11}^n}{u}\left(c_{11}+2n-\frac{c_{12}(2n+1)}{u-v}\right)+\frac{ns_{22}^n}{uv}\left(\frac{c_{22}(2n+1)-2n(c_{11}-1)-c_{21}u}{u+v}+\frac{2nc_{12}(3n+1)}{u^2-v^2}\right),\\
            (1,2):&\quad \dfrac{c_{12}s_{11}^n}{u-v}-\dfrac{ns_{22}^n}{uv}\left(c_{22}+2n+\dfrac{c_{12}2n}{u-v}\right),\\
            (2,1):&\quad
                 \frac{s_{22}^n}{u}\Biggl(\frac{2n(c_{11}-1)-c_{22}(2n+1)+c_{21}u}{u+v}-\frac{2nc_{12}(3n+1)}{u^2-v^2}-\frac{n}{v}\left(c_{11}+2n-\frac{c_{12}(2n+1)}{u-v}\right)\Biggr),\\
            (2,2):&\quad \dfrac{s_{22}^n}{u}\left(c_{22}+2n+\dfrac{nc_{12}(2v-u)}{v(u-v)}\right).
        \end{align*}

        For $n=0$, the diagonal entries imply that $\Im(c_{12})=(u-v)\Im(c_{11})$ and $c_{22}\in\mathbb{R}$. For $n\geq 1$, they imply $\Im(c_{21})=\frac{2\Im(c_{12})}{u(u-v)}\left(3n^2-\frac{s_{11}^nv(u+v)}{s_{22}^n}\right)$ and one of the following conditions: $c_{12}\in\mathbb{R}$ or $v=u/2$. Both conditions hold, yet we will exclusively verify them by assuming $v=u/2$ as our starting point, because the analysis starting from $c_{12}\in\mathbb{R}$ is analogous.

        \

        Suppose $v=u/2$. The entries $(1,2)$ and $(2,1)$ lead to the following relation between the diagonal entries of $S_n$:
        \begin{equation}\label{s11ueq2v}
            s_{11}^n=\frac{s_{22}^n\Bigl( 12\overline{c_{12}}n^2+\bigl(2c_{12}+u(c_{22}-c_{11}-2)\bigr)n +u(c_{21}u-c_{22})\Bigr)}{3u^2\overline{c_{12}}}, \quad n\geq 0.
        \end{equation}

        Next we use the fact that $S_n=S_{n-1}A_n$ for all $n\geq 1$. For this, we take  $k=n-2$ in Lemma \ref{FnGn} to obtain $T_{n-2}^n=\begin{pmatrix}
            f_3^{n} & g_3^{n}\\
            f_4^{n} & g_4^{n}
        \end{pmatrix}$, where
        \begin{align*}
            f_3^{n}&=-\dfrac{(n-1)\bigl((c_{11}+n-2)f_1^{n}+c_{12}f_2^{n}\bigr)+ng_3^{n}}{2u},\\
            f_4^{n}&=2\frac{(n-2)f_3^{n}-ng_4^{n}-(n-1)\bigl(c_{21}f_1^{n}+(c_{22}+n-2)f_2^{n}\bigr)}{5u},\\
            g_3^{n}&=\dfrac{2n(n-1)c_{12}}{3u^2}\left(\frac{2(n-1)c_{12}}{u}+2c_{11}+c_{22}+3n-5\right),\\
            g_4^{n}&=\frac{(n-2)g_3^{n}}{2u}+\frac{n(n-1)}{u^2}\Biggl(c_{12}c_{21}+(c_{22}+n-2)\left(\frac{(n-1)c_{12}}{u}+\frac{c_{22}+n-1}{2}\right)\Biggr),
            \end{align*}
        and $f_1^{n},f_2^{n}$ denote respectively the entries $(1,1),(2,1)$ of $T_{n-1}^n$.
        
        Now, we calculate $A_n=T_{n-2}^n-T_{n-1}^{n+1}-T_{n-1}^nB_n$, for all $n\geq 1$. The entries $(1,2)$ and $(2,1)$ of $S_{n-1}A_n$ are given by $s_{22}^{n-1}p(n)$ and $s_{22}^{n-1}q(n)$, where $p,q$ are third-degree polynomials in $n$, whose leading coefficients are respectively
        \[
        \frac{8c_{12}^2+4uc_{12}(c_{22}-c_{11}+1)-6u^2}{3u^2} \quad \text{and}\quad \frac{8c_{12}^2+4uc_{12}(c_{22}-c_{11}+1)-10u^2}{5u^2}.
        \]
        Due to the symmetry of $S_n$, we have $c_{12}=\frac{u(c_{11}-c_{22}+1)}{2}$. The previous conditions $\Im(c_{11})=2\Im(c_{12})$ and $\Im(c_{21})=12\Im(c_{12})\left(\frac{n^2}{u^2}-\frac{s_{11}^n}{4s_{22}^n}\right)$ imply that $c_{11},c_{21}\in\mathbb{R}$.
        
        Comparing the other coefficients of the polynomials $p,q$, we derive $c_{21}=\frac{2c_{22}-3}{2u}$, $c_{11}=2-c_{22}$ and $c_{12}=\frac{u(1-2c_{22})}{2}$. Also, \eqref{s11ueq2v} implies $c_{22}>1/2$.          

            \bigskip
       
            Now, we look for a symmetric positive definite matrix $W(t)$ that verifies the first and second-order symmetry equations
            \[
            2(WF_2)'=WF_1+F_1^*W 
            \quad \text{and}\quad 
            (WF_2)''=(WF_1)'+F_0^*W-WF_0,
            \]
            where $F_2(t)=tI$, $F_1(t)=C-tU$ and $F_0=-V$. Considering $W(t)=e^{-tu} \begin{pmatrix}
            w_{11}(t) & w_{12}(t) \\
             \overline{w_{12}}(t) & w_{22}(t)
            \end{pmatrix}$, the first equation turns into
            \begin{equation}           
             \label{edo1vu/2} 
            0=\begin{pmatrix}
                    2t w'_{11}+2(c_{22}-1) w_{11}+2\left(t+\frac{3-2c_{22}}{2u}\right) \Re(w_{12})    
                    & 2t w'_{12}+\left(t+\frac{3-2c_{22}}{2u}\right) w_{22}+\left(c_{22}-\frac{1}{2}\right)w_{11} \\[.5em]
                    2t \overline{w'_{12}}+\left(t+\frac{3-2c_{22}}{2u}\right) w_{22}+\left(c_{22}-\frac{1}{2}\right)w_{11}
                    & 2t w'_{22}+2(1-c_{22}) w_{22}+(2c_{22}-1) \Re(w_{12})
                    \end{pmatrix},
                    \end{equation}
            and the second equation leads to        
              \begin{align}
                 \label{edo2vu/2_11}
                      t w''_{11}+(c_{22}-tu) w'_{11}+u(1-c_{22})w_{11}+\left(t+\frac{3-2c_{22}}{2u}\right) w'_{12}+\left(c_{22}-\frac{1}{2}-tu\right)w_{12} =0,  \\ \label{edo2vu/2_12}  
                     t w''_{12}+(2-c_{22}-tu) w'_{12}+u\left(c_{22}-\frac{1}{2}\right)w_{12}+u\left(c_{22}-\frac{1}{2}\right) w'_{11}+u^2\left(\frac{1}{2}-c_{22}\right)w_{11}=0,  \\ \label{edo2vu/2_21}
                     t \overline{w''_{12}}+(c_{22}-tu)\overline{w'_{12}}+u\left(\frac{1}{2}-c_{22}\right) \overline{w_{12}}+\left(t+\frac{3-2c_{22}}{2u}\right)w'_{22}+\left(c_{22}-\frac{1}{2}-tu\right)w_{22}=0, \\
                    \label{edo2vu/2_22} t w''_{22}+(2-c_{22}-tu) w'_{22}+u(c_{22}-1) w_{22}+u\left(c_{22}-\frac{1}{2}\right)\overline{w'_{12}}+u^2\left(\frac{1}{2}-c_{22}\right)\overline{w_{12}}=0. \end{align}      
    
            Since $w_{11}$ and $w_{22}$ are real, the entry $(1,2)$ of \eqref{edo1vu/2} implies that $\Im(w_{12})$ is constant. The imaginary part of \eqref{edo2vu/2_11} leads to $\Im(w_{12})=0$. Then $w_{12}$ is real.
                  
            From entry $(2,2)$ of \eqref{edo1vu/2} we have the equation $w_{12}(t)=\dfrac{2\bigl((c_{22}-1)w_{22}-t w'_{22}\bigr)}{u(2c_{22}-1)}$. Replace it in \eqref{edo2vu/2_21} to get the following third-order differential equation
            \[
            4t^2 w'''_{22}-4t(tu-3)w''_{22}+\bigl(3-4(c_{22}+1)tu\bigr)w'_{22}+u(2c_{22}-1)(2tu-1)w_{22}=0,
            \] 
            whose integrable solution is
            \begin{equation*}
            w_{22}(t)=e^{-\sqrt{2(2c_{22}-1)tu}}t^{-1/2}\left(k_1e^{2\sqrt{2(2c_{22}-1)tu}}-\frac{k_2}{\sqrt{2(2c_{22}-1)u}}\right),
            \end{equation*}
            where $k_1,k_2\in\mathbb{R}$ and $u>0$. Then we have
            \[
            w_{12}(t)=e^{-\sqrt{2(2c_{22}-1)tu}}t^{-1/2}
            \left(k_1e^{2\sqrt{2(2c_{22}-1)tu}}\frac{\sqrt{2c_{22}-1}-\sqrt{2tu}}{u\sqrt{2c_{22}-1}}-k_2\frac{\sqrt{2c_{22}-1}+\sqrt{2tu}}{\sqrt{2u^3}(2c_{22}-1)}\right).
            \]
            Since $S_0=\int W(t)dt$ is diagonal, then
            \[
            0=\int_{0}^{\infty} w_{12}(t)dt=-\frac{2c_{22}\bigl(k_2+k_1\sqrt{2u(2c_{22}-1)}\bigr)}{u^2(2c_{22}-1)^2}.
            \]            
            Hence $k_2=-k_1\sqrt{2u(2c_{22}-1)}$. 
            
            Subsequently, solving the differential equation for $w_{11}$ of the entry $(1,1)$ of \eqref{edo1vu/2} we have
            \begin{equation*}
                w_{11}(t)=k_3t^{1-c_{22}}+2k_1\frac{\cosh \bigl(\sqrt{2(2c_{22}-1)tu}\bigr)(2c_{22}-1+2tu)-2\sinh\bigl(\sqrt{2(2c_{22}-1)tu}\bigr)\sqrt{2(2c_{22}-1)tu}}{u^2(2c_{22}-1)\sqrt{t}}.
            \end{equation*}
            Finally, this function verifies the first and second-order differential equations if $k_3=0$. No generality is lost if we take $k_1=1/2$, which gives the matrix weight of the statement.
            
            Following the idea of Theorem \ref{Teovequ} it can be proved that this weight and the corresponding family of monic polynomials do not reduce.     
    \end{proof}
    
    \bigskip
     \section{\texorpdfstring{$U$ is a scalar matrix}{TEXT}}\label{sec5}

    In this section, we consider the case in which the Jordan canonical form of $U$ is scalar. In this context, we identify  all the $W$-symmetric differential operators $D$ of the form:    \begin{equation}\label{Desc}
            D=tI \partial^2
            +   \begin{pmatrix}
                    c_{11} - u t  & c_{12}\\
                    c_{21}  & c_{22} - u t 
                \end{pmatrix} \partial
            -   \begin{pmatrix}
                    v & 0 \\
                    v_{21} & 0
                \end{pmatrix},
        \end{equation}
    where $u,v\in\mathbb{R}$, $v_{21}, c_{ij}\in\mathbb{C}$ and $t\in(0,\infty)$, such that the family of monic orthogonal polynomials $\{P_n\}_{n\in\mathbb{N}_0}$ verifies $DP_n=P_n\Delta_n$ for all $n\geq 0$, where
    \[
        \Delta_n=\begin{pmatrix}
            \lambda_n & 0\\
            \nu_n & \mu_n
        \end{pmatrix}
        =\begin{pmatrix}
            -nu-v & 0\\
            -v_{21} & -nu
        \end{pmatrix}.
    \]

    \begin{proposition}
    If the second-order differential operator \eqref{Desc} satisfies $DP_n=P_n\Delta_n$, then $v_{21}=0$, $u\neq 0$ and $v>0$.
    \end{proposition}

    \begin{proof}
    By reiterating the argument of Proposition 4.1, we derive that $v_{21}=0$.  
    
    Suppose $u=0$. Equation \eqref{Fn} for $k=n-1$ becomes $
     \begin{pmatrix}
                0\\
                -vf_2^{n}
            \end{pmatrix} 
            =
            n\begin{pmatrix}
                c_{11}+n-1\\
                c_{21}
            \end{pmatrix},
    $
    wich is a contradiction since $c_{11}$ is constant.

    Now suppose $v=0$. Equations \eqref{Fn} and \eqref{Gn} for $k=n-1$ and $k=n-2$ lead to the expressions
    \begin{align*}
        T_{n-1}^n = \dfrac{-n(C + (n-1)I)}{u} \quad \text{and}\quad T_{n-2}^n =\dfrac{n(n-1)(C+(n-1)I)(C+(n-2)I)}{2u^2}.
    \end{align*}
        
    Now, by Lemma \ref{AnBn}, we have that $B_n=\frac{C+2nI}{u}$ for all $n\geq 0$ and $A_n=\frac{C+(n-1)I}{u^2}$ for all $n\geq 1$.

        Consider the non-scalar matrix $T=\begin{pmatrix}
               0 &c_{12}\\
                c_{21} & c_{22}-c_{11}
            \end{pmatrix}$. This matrix commutes with $C$. Thus, it commutes with $A_n$ and $B_n$ for all $n$. Consequently, $T$ commutes with $P_n$ for all $n\geq 0$. Now, by Theorem \ref{Pnreducible}, we have that this case reduces. 

           Finally, notice that no generality is lost by considering $v>0$, as 
            \[
        \begin{pmatrix}
            0 & 1\\
            1 & 0
            \end{pmatrix}^{-1} D \begin{pmatrix}
                0 & 1\\
                1 & 0
            \end{pmatrix}+vI=t\partial^2+\begin{pmatrix}
                c_{22}-tu & c_{21}\\
                c_{12} & c_{11}-tu
            \end{pmatrix}\partial-\begin{pmatrix}
                -v & 0\\0 & 0
            \end{pmatrix},
        \]
         remains an equivalent operator.
    \end{proof}
        
 \begin{theorem}\label{Teovequesc}
        Let $D$ be the differential operator \eqref{Desc} with $v=|u|$. Then, $D$ is $W$-symmetric for some weight $W$ if and only if is
        \[
            D=tI\partial^2+\begin{pmatrix}
                   c_{22}+2-tu & 0\\
                   c_{21} & c_{22}-tu
                \end{pmatrix}\partial-\begin{pmatrix}
                    u & 0\\ 0 & 0
                \end{pmatrix},
        \]
        where $c_{22}>0$ and $u>0$. Moreover, $D$ is symmetric with respect to the weight function
        \[
            W(t)=e^{-tu}t^{c_{22}-1}\begin{pmatrix}
           \gamma t^2+\dfrac{|c_{21}|^2(c_{22}-2tu)}{4c_{22}}
           & \dfrac{\overline{c_{21}}(tu-c_{22})}{2c_{22}}\\
           \dfrac{c_{21}(tu-c_{22})}{2c_{22}}
           &1
            \end{pmatrix}, \quad \gamma>\dfrac{|c_{21}|^2u^2}{4c_{22}^2}.
        \]
    \end{theorem}

    \begin{proof}
    We consider separately the two cases $v=u$ and $v=-u$.

    First, suppose $v=u$. By Lemma \ref{FnGn} for $k=n-1$ and $k=n-2$ we have that $c_{12}=0$ and
    \begin{align*}
    T_{n-1}^n&=\begin{pmatrix}
            \frac{-n(c_{11}+n-1)}{u} & g_1^n\\
            \frac{-nc_{21}}{2u} & \frac{-n(c_{22}+n-1)}{u}
        \end{pmatrix},\quad n\geq 0,\\[.5em]
    T_{n-2}^n&=(n-1)\begin{pmatrix}
           \frac{n(c_{11}+n-1)(c_{11}+n-2)}{2u^2} 
           & -\frac{g_1^n(c_{11}+n-2)}{u}\\
            \frac{nc_{21}(2c_{11}+c_{22}+3n-4)}{6u^2}
           &\frac{n(c_{22}+n-1)(c_{22}+n-2)-c_{21}ug_1^n}{2u^2} 
        \end{pmatrix}, \quad n\geq 1.
    \end{align*}
    Since $C$ is not a diagonal matrix (otherwise $D$ would reduce), we have $c_{21}\neq 0$. 
    Now, by Lemma \ref{AnBn} we have
        \begin{equation*}
        B_{n}=\begin{pmatrix}
            \frac{c_{11}+2n}{u} & g_1^n-g_1^{n+1}\\[.5em]
            \frac{c_{21}}{2u} & \frac{c_{22}+2n}{u}
        \end{pmatrix},\quad
        A_n=\begin{pmatrix}
            \frac{2n(c_{11}+n-1)-c_{21}ug_1^n}{2u^2} 
            &\frac{(c_{11}-c_{22}-2)g_1^n}{u}\\
            \frac{nc_{21}(c_{22}-c_{11}+2)}{6u^2}
            &\frac{2n(c_{22}+n-1)+c_{21}ug_1^n}{2u^2} 
        \end{pmatrix}.
        \end{equation*}
        
        By Lemma \ref{Sn}, we have that $S_n$ is diagonal for all $n\geq 0$, because $\nu_n=0$. The fact that $S_nB_n$ is Hermitian for all $n\geq 0$ implies $c_{11},c_{22}\in\mathbb{R}$. The identity $S_n=S_{n-1}A_n$ implies that $A_n$ is also diagonal. Since $c_{21}\neq 0$, we have $c_{11}=c_{22}+2$. It can be shown that this family does not reduce, because there is no non-scalar matrix $T$ which commutes with $A_n$ and $B_n$.

        \bigskip
        
        Now we solve the first-order symmetry equation $2(WF_2)'=WF_1+F_1^*W$, where $F_2(t)=tI$, $F_1(t)=C-tU$ and $W(t)=(w_{ij}(t))$ is a symmetric positive definite matrix. It is given by
        \[
        0=\begin{pmatrix}
            tw_{11}'+(tu-c_{22}-1)w_{11}-\Re(c_{21}w_{12})    
            & tw_{12}'+(tu-c_{22})w_{12}-\frac{\overline{c_{21}}w_{22}}{2}\\[.5em]
            t\overline{w_{12}}'+(tu-c_{22})\overline{w_{12}}-\frac{c_{21}w_{22}}{2}
            & tw_{22}'+(tu-c_{22}+1)w_{22}
            \end{pmatrix}.
        \]
        The solution of this system is
        \[
        W(t)=e^{-tu}t^{c_{22}-1}\begin{pmatrix}
            k_3t^2-\Re(k_2c_{21})t+\frac{k_1|c_{21}|^2}{4} 
            & k_2 t-\frac{k_1\overline{c_{21}}}{2}\\[.5em]
            \overline{k_2}t-\frac{k_1c_{21}}{2}
            & k_1
        \end{pmatrix},
        \]
        where $k_1,c_{22}>0$ for $w_{22}$ to be positive and integrable. Since 
        $S_0=\int W(t)dt$ is diagonal, then
        \[
        0=\int_{0}^{\infty} w_{12}(t)dt=\dfrac{\Gamma(c_{22})(2k_2c_{22}-k_1u\overline{c_{21}})}{2u^{c_{22}+1}}.  
        \]
        It follows that  $k_2=\frac{k_1u\overline{c_{21}}}{2c_{22}}$. We can consider without loss of generality $k_1=1$ and then $W(t)$ becomes the weight in the statement. It is straightforward that this weight verifies the second-order symmetry equation as well. And that does not reduce.

        \bigskip
        
        Now, suppose $v=-u$. By Lemma \ref{FnGn} for $k=n-1$ and  Lemma \ref{AnBn} we have that $c_{21}=0$ and
        \[
        T_{n-1}^n=\begin{pmatrix}
            \frac{(c_{11}+n-1)}{u} & \frac{-n c_{12}}{2u}\\[.5em]
            f_2^{n} & \frac{-n(c_{22}+n-1)}{u}
        \end{pmatrix}, \quad B_n=\begin{pmatrix}
                    \frac{c_{11}+2n}{u} & \frac{c_{12}}{2u}\\[.5em]
                    f_2^{n}-f_2^{n+1} & \frac{c_{22}+2n}{u}
                \end{pmatrix}, \quad n\geq 0.
            \]
        The fact that $S_n$ is diagonal and $S_nB_n$
        is Hermitian for all $n\geq 0$, implies that $c_{11}$ and $c_{22}$ are real numbers.
        
        Now the entry $(1,1)$ of the first-order symmetry equation is $ 2(tw_{11}'+w_{11})=2(c_{11}-tu)w_{11}$,
        whose solution is a multiple of $e^{-tu}t^{c_{11}-1}$. Since $v>0$ and $v=-u$ we have that $w_{11}$ is not integrable in $(0,\infty)$. Therefore, there is no weight for $v=-u$.
    \end{proof}
        
\bigskip  

    \begin{proposition}
       If $D$ is the operator \eqref{Desc} with $v\neq |u|$, then there is no weight such that the family of monic orthogonal polynomials $\{P_n\}_{n\in\mathbb{N}_0}$ satisfies $DP_n=P_n\Delta_n$ for all $n\geq 0$.
    \end{proposition}

    \begin{proof}
    From Lemma \ref{FnGn} for $k=n-1$ and Lemma \ref{AnBn}, it follows that
     \begin{equation*}
        T_{n-1}^n=\begin{pmatrix}
            \frac{-n(c_{11}+n-1)}{u} & \frac{-nc_{12}}{u-v}\\[.5em]
            \frac{-nc_{21}}{u+v} & \frac{-n(c_{22}+n-1)}{u}
        \end{pmatrix}, \quad B_n=\begin{pmatrix}
            \frac{c_{11}+2n}{u} & \frac{c_{12}}{u-v}\\[.5em]
            \frac{c_{21}}{u+v} & \frac{c_{22}+2n}{u}
        \end{pmatrix}, \quad n\geq 0.
        \end{equation*}
   
        The fact that $S_n$ is diagonal and $S_nB_n$ is Hermitian for all $n\geq 0$ implies that $c_{11},c_{22}\in\mathbb{R}$ and 
        \[
        s_{11}^n\overline{c_{12}}(u+v)=s_{22}^nc_{21}(u-v).
        \]  Since $s_{11}^n$, $s_{22}^n$ are positive and $D$ is non-diagonal, we have that both $c_{12}$ and $c_{21}$ are non-zero numbers. Consequently, we derive the relation:
        \begin{equation}\label{s1esc}
        s_{11}^n=\frac{c_{21}(u-v)}{\overline{c_{12}}(u+v)}s_{22}^n, \quad n\geq 0.
        \end{equation}     

        By Lemma \ref{FnGn} for $k=n-2$ and Lemma \ref{AnBn} we obtain
        \begin{align*}
        T_{n-2}^n&=\begin{pmatrix}
            \frac{n(n-1)}{2u}\left(\frac{(c_{11}+n-1)(c_{11}+n-2)}{u}+\frac{c_{12}c_{21}}{u+v}\right) & g_3^n\\
            f_4^n & \frac{n(n-1)}{2u}\left(\frac{(c_{22}+n-1)(c_{22}+n-2)}{u}+\frac{c_{12}c_{21}}{u-v}\right)
        \end{pmatrix},\\
            A_n&=\begin{pmatrix}
                    n\left(\frac{c_{11}+n-1}{u^2}+\frac{c_{12}c_{21}v}{u(u^2-v^2)}\right)
                    &\frac{c_{12}n\bigl(c_{11}+c_{22}+3n-1)\bigr)}{u(u-v)}+g_3^n-g_3^{n+1}\\[.5em]
                    \frac{c_{21}\bigl(c_{11}+c_{22}+3n-1)\bigr)}{u(u+v)}+f_4^n-f_4^{n+1}
                    & n\left(\frac{c_{22}+n-1}{u^2}-\frac{c_{12}c_{21}v}{u(u^2-v^2)}\right)      \end{pmatrix},
        \end{align*} 
        where 
        \begin{equation} \begin{split}\label{condg3}
            (v-2u)g_3^n=c_{12}n(n-1)\left(\frac{c_{11}+n-2}{v-u}+\frac{c_{22}+n-1}{u}\right), \\(v+2u)f_4^n=c_{21}n(n-1)\left(\frac{c_{11}+n-1}{u}+\frac{c_{22}+n-2}{u+v}\right).
        \end{split}
        \end{equation}
            
        From $S_n=S_{n-1}A_n$ and $S_n$ being diagonal, it follows that $A_n$ is diagonal. 
        Solving the recurrence relation of the entries $(1,2)$ and $(2,1)$ with the initial conditions $g_3^n(1)=0$ and $f_4^n(1)=0$ we have
        \begin{equation*}
            g_3^n=\frac{c_{12}\,  n (n-1)( c_{11} + c_{22}+ 2n-2)}{2 u (u - v)}, \quad f_4^n=\frac{c_{21}\,  n (n-1)\bigl( c_{11} + c_{22}+ 2n-2)\bigr)}{2 u (u + v)},\quad n\geq 2.   
        \end{equation*}
        Then equations \eqref{condg3} imply $c_{11}=c_{22}+\frac{2u}{v}$. 

        Now solve the recurrence relation of $S_n=S_{n-1}A_n$. The entries $(1,1)$ and $(2,2)$ lead respectively to
        \[
        s_{22}^n=\frac{s_{22}^0 n!}{u^{2n}}\left(c_{22}-\frac{uv^2(2-c_{12}c_{21})-2u^3}{v(u^2-v^2)}\right)_n \quad \text{and} \quad s_{22}^n=\frac{s_{22}^0 n!}{u^{2n}}\left(c_{22}-\frac{c_{12}c_{21}uv}{u^2-v^2}\right)_n.
        \]
        The equality holds if $c_{12}=\frac{v^2-u^2}{c_{21}v^2}$, but thus $s_{11}^n<0$.
    \end{proof}

    \bigskip   


\section{\texorpdfstring{$U$}{TEXT} has two different eigenvalues}\label{sec6}

    In this section, we examine the third and last Jordan canonical form of matrix $U$. Specifically, we study operators $D$ of the form
    \begin{equation}\label{Ddiagonal}
    D=t I \partial^2
            +   \begin{pmatrix}
                    c_{11} - t u_1 & c_{12}\\
                    c_{21}  & c_{22} - t u_2
                \end{pmatrix} \partial
            -   \begin{pmatrix}
                    v & 0 \\
                    v_{21} & 0
                \end{pmatrix}, 
    \end{equation}
    where $t\in(0,\infty)$, $c_{ij},v_{21}\in\mathbb{C}$, $u_1,u_2,v\in\mathbb{R}$ and $u_1\neq u_2$. If $\{P_n\}_{n\in\mathbb{N}_0}$ is a sequence of monic orthogonal polynomials satisfying $DP_n=P_n\Delta_n$ for all $n\geq 0$, then the eigenvalues are 
    \[
        \Delta_n=\begin{pmatrix}
                       \lambda_n  & 0\\
                        \nu_n & \mu_n
                \end{pmatrix}
                =\begin{pmatrix}
                    -nu_1-v & 0\\
                    -v_{21} & -nu_2
                \end{pmatrix}.
    \]
    
    \begin{proposition}
    If the second-order differential operator \eqref{Ddiagonal} satisfies $DP_n=P_n\Delta_n$ for all $n\geq 0$, then $u_1$ and $u_2$ are non-zero.
    \end{proposition}

    \begin{proof}
    Suppose $u_1=0$. Equations \eqref{Fn}  and \eqref{Gn} for $k=n-1$ are given by 
        \begin{align}
            \begin{pmatrix}
                0\\
                v_{21}f_1^{n}+ \bigl(nu_2-u_2-v\bigr)f_2^{n}
            \end{pmatrix} &=n\begin{pmatrix}
                c_{11}+n-1\\
                c_{21}
            \end{pmatrix}+v_{21}\begin{pmatrix}
                g_1^n\\g_2^{n}
            \end{pmatrix},\label{F1ndiagonal}\\
            \begin{pmatrix}
                \bigl(v-nu_2\bigr)g_1^n\\
                v_{21}g_1^n-u_2 g_2^{n}
            \end{pmatrix}&=n\begin{pmatrix}
                c_{12}\\c_{22}+n-1
            \end{pmatrix}\label{G1ndiagonal}.
        \end{align}
    From the first entry of \eqref{F1ndiagonal} we see that since $c_{11}$ is constant, then $v_{21}\neq0$. Thus, 
    $
        g_1^n=\frac{-n(c_{11}+n-1)}{v_{21}}.
    $
    Then the first entry of \eqref{G1ndiagonal} becomes 
    $\frac{(c_{11}+n-1)(v-nu_{2})}{v_{21}}=-c_{12}$ for all $ n\geq 1$, leading to a contradiction. The case $u_2=0$ is entirely analogous.    
    \end{proof}

    \bigskip

    \begin{remark}\label{hipextra}
    From Lemma \ref{Sn} we have the relation $s_{22}^nv_{21}=(\lambda_n-\mu_n)\overline{s_{12}^n}$ for all $n\geq 0$. It is clear that the difference $\lambda_n-\mu_n$ cannot vanish for infinitely many $n$. We dismiss the possibility that it does for some particular value of $n$. Also, the matrices 
    \[
    \lambda_nI-\Delta_{n-k}=\begin{pmatrix}
        -ku_1 & 0\\
        v_{21} & \lambda_{n}-\mu_n-ku_2
    \end{pmatrix}
    \quad
    \text{and}
    \quad
    \mu_nI-\Delta_{n-k}=\begin{pmatrix}
        \mu_n-\lambda_n-ku_1 & 0\\
        v_{21} & -ku_2
    \end{pmatrix}, \quad k=1,2,
    \] 
    are invertible for infinitely many $n$. We consider they are for all $n\geq 0$. 
    
    Both assumptions are equivalent to ask $\lambda_n-\mu_n\neq ku_1,-ku_2$ for all $n\geq 0$ and $k=0,1,2$. Then Lemmas \ref{Sn} and \ref{FnGn} lead to the following result.
\end{remark}

    \begin{corollary}
        If $\lambda_n-\mu_n\neq ku_1,-ku_2$ for all $n\geq 0$ and $k=0,1,2$, then 
        the coefficients $T_{n-1}^n$ and $T_{n-2}^n$ of the monic polynomial $P_n(t)$ are given by
         \begin{align*}
            T_{n-1}^n&=\begin{pmatrix}
                -\frac{n(c_{11}+n-1)+v_{21}g_1^n}{u_1}
                & \frac{nc_{12}}{\mu_n-\lambda_{n-1}}\\[.5em]
                \frac{nc_{21}+v_{21}(g_2^{n}-f_1^{n})}{\lambda_n-\mu_{n-1}}
                & \frac{v_{21}g_1^n-n(c_{22}+n-1)}{u_2}
            \end{pmatrix},\quad n\geq 0\\
            T_{n-2}^n&=(n-1)\begin{pmatrix}
            - \frac{(c_{11}+n-2)f_1^{n}+c_{12}f_2^{n}+\frac{v_{21}g_3^{n}}{n-1}}{2u_1} & \frac{(c_{11}+n-2)g_1^n+c_{12}g_2^{n}}{\mu_n-\lambda_{n-2}}\\[.5em]
            \frac{(c_{22}+n-2)f_2^{n}+c_{21}f_1^{n}+\frac{v_{21}(g_4^{n}-f_3^{n})}{n-1}}{\lambda_n-\mu_{n-2}} & \frac{\frac{v_{21}g_3^{n}}{n-1}-(c_{22}+n-2)g_2^{n}+c_{21}g_1^n}{2u_2}
    \end{pmatrix},\quad n\geq 1.
         \end{align*}
    \end{corollary}

    \begin{theorem}\label{teodiagonal}
    There is no irreducible weight $W$ such that the operator $D$ given by \eqref{Ddiagonal} is $W$-symmetric.
    \end{theorem}

    \begin{proof}        
        Using Lemma \ref{AnBn}, we find that
        \[
        B_n=\begin{pmatrix}
            \frac{c_{11}+2n}{u_1}+\frac{c_{12}v_{21}(v-u_1)}{u_1(\mu_n-\lambda_{n-1})(\mu_{n+1}-\lambda_n)}
            &
            \frac{c_{12}(u_1-v)}{u_1(\mu_n-\lambda_{n-1})(\mu_{n+1}-\lambda_n)}\\[1em]
            \frac{(u_2+v)c_{21}}{(\lambda_{n+1}-\mu_n)(\lambda_n-\mu_{n-1})}+v_{21}\left(\frac{g_2^{n}-f_1^{n}}{\lambda_n-\mu_{n-1}}-\frac{g_2(n+1)-f_1^{n}}{\lambda_{n+1}-\mu_{n}}\right)
            &
            \frac{c_{22}+2n}{u_2}+\frac{c_{12}v_{21}(u_1-v)}{u_2(\mu_n-\lambda_{n-1})(\mu_{n+1}-\lambda_n)}
        \end{pmatrix},\quad n\geq 0.
        \]

        The identity \eqref{nun} gives $s_{12}^n=\frac{s_{22}^n\overline{v_{21}}}{\lambda_n-\mu_n}$. Then, entry $(2,2)$ of the Hermitian matrix $S_nB_n$ is  
        \[
        \frac{s_{22}^n}{u_2}\left(c_{22}+2n+\frac{v_{21}c_{12}(u_1-v)\bigl(n(u_2-u_1)-v+u_2\bigr)}{(\mu_n-\lambda_{n-1})(\mu_{n+1}-\lambda_n)(\lambda_n-\mu_n)}\right), \quad n\geq 0.
        \]
        This implies $c_{22}\in\mathbb{R}$ and $(v-u_1)\Im(c_{12}v_{21})=0$. Since entries $(1,2)$ and $(2,1)$ are conjugated, we find the relation
        \begin{equation}\label{s1diag}
        s_{11}^n\overline{c_{12}}(u_1-v)=u_1s_{22}^n(\mu_n-\lambda_{n-1})(\mu_{n+1}-\lambda_n)\left(\frac{v_{21}(b_{11}^n-b_{22}^n)}{\lambda_n-\mu_n}-b_{21}^n\right),
        \end{equation}
        where $b_{ij}^n$ denotes the entry $(i,j)$ in $B_n$. Here we divide the calculations according to the value of $\overline{c_{12}}(u_1-v)$. 

        \
        
        First, suppose $\overline{c_{12}}(u_1-v) \neq 0$. Then we have that $c_{12}v_{21}\in\mathbb{R}$ and
        \[
        s_{11}^n=\frac{u_1s_{22}^n(\mu_n-\lambda_{n-1})(\mu_{n+1}-\lambda_n)}{\overline{c_{12}}(u_1-v)}\left(\frac{v_{21}(b_{11}^n-b_{22}^n)}{\lambda_n-\mu_n}-b_{21}^n\right).
        \]

        On the other hand, the entries $(1,2)$ and $(2,2)$ in the identity $S_n=S_{n-1}A_n$ imply that 
        \[
            \dfrac{\overline{v_{21}}s_{22}^n}{\lambda_n-\mu_n}=s_{11}^{n-1}a_{12}^n+s_{12}^{n-1}a_{22}^n,\quad \text{and} \quad
            s_{22}^n=\overline{s_{12}^{n-1}}a_{12}^n+s_{22}^{n-1}a_{22}^n,
        \]
        where $a_{ij}^n$ is the entry $(i,j)$ of the matrix $A_n$. 
        Thus, we obtain the following
        \[
        \frac{\overline{v_{21}}\left(\frac{v_{21}a_{12}^n}{\lambda_{n-1}-\mu_{n-1}}+a_{22}^n\right)}{\lambda_n-\mu_n}=\frac{\left(\frac{v_{21}(b_{11}^{n-1}-b_{22}^{n-1})a_{12}^n}{\overline{b_{12}^{n-1}}}+\overline{v_{21}a_{22}^n}\right)}{\lambda_{n-1}-\mu_{n-1}}+\frac{b_{21}^{n-1}a_{12}^n}{\overline{b_{12}^{n-1}}}.
        \]
        This is equivalent to an expression of the form $0=p(n)/q(n)$, where $\deg(p)=9$ and $\deg(q)=7$. The leading coefficient of $p$ is $c_{12}(u_1-v)(u_1-u_2)^8v_{21}$, which only vanishes when $v_{21}=0$. The coefficient of $n^7$ is $c_{12}(u_1-v)(u_1-u_2)^6u_1u_2c_{21}(u_2+v)$, which equals zero for $c_{21}(u_2+v)=0$. Then $b_{21}^n=$ for all $n\geq 0$ and thus relation \eqref{s1diag} becomes zero, which contradicts the positive definiteness of $S_n$.
        
        \   
        
        Now, suppose $\overline{c_{12}}(u_1-v)=0$. The relation \eqref{s1diag} becomes
        \[
        0=\frac{v_{21}(b_{11}^n-b_{22}^n)}{\lambda_n-\mu_n}-b_{21}^n,
        \]
        whose right side is a fourth-degree polynomial in $n$ with leading coefficient $(u_1-u_2)^4v_{21}$. Then $v_{21}=0$ and 
        \[
         0=b_{21}^n=\frac{(u_2+v)c_{21}}{(\lambda_{n+1}-\mu_n)(\lambda_n-\mu_{n-1})}, \quad n\geq 1.
        \]
        For $n=0$ the Hermicity of $S_nB_n$ implies 
        \begin{equation}\label{s0b0_diag}
        \frac{c_{12}s_{11}^0}{u_2-v}=\frac{\overline{c_{21}} s_{22}^0}{u_1+v}.
        \end{equation} 
        Since $S_0$ is positive definite and $C$ is non-diagonal, we have $u_1=-u_2=v$.
        Also, the entry $(1,1)$ of $S_nB_n$ given by $\frac{s_{11}^n(c_{11}+2n)}{v}$, implies $c_{11}\in\mathbb{R}$. 
                
        By Lemma \eqref{AnBn}, we have that 
        
        \[
        A_n=\begin{pmatrix}
            \frac{4n(c_{11}+n-1)-c_{12}c_{21}}{4v^2}
            &\frac{c_{12}(2-c_{11}-c_{22})}{2v^2(4n^2-1)}\\
            \frac{c_{21}(2-c_{11}-c_{22})}{2v^2(4n^2-1)}
            &\frac{4n(c_{22}+n-1)-c_{12}c_{21}}{4v^2}
        \end{pmatrix}.
        \]
            
        As $S_n=S_{n-1}A_n$ and $S_n$ is diagonal, it follows that $A_n$ is diagonal as well. Then, $c_{22}=2-c_{11}$.
        
        To simplify the following calculations, consider the operator
        \[M^{-1} D M= t I\partial^2 +\begin{pmatrix}
                c_{11}-vt & 1\\
                c_{21}c_{12} & 2-c_{11}+vt
             \end{pmatrix}\partial -\begin{pmatrix}
                 v & 0\\
                 0 & 0
             \end{pmatrix},
        \]
        where $M=\begin{pmatrix}
            c_{12} & 0\\
            0 & 1
        \end{pmatrix}$. From (\ref{s0b0_diag}) and $S_n$ positive definite, we deduce $c_{21}c_{12}=-\frac{\overline{c_{12}}c_{12} s_{11}^0}{s_{22}^0}<0$. In what follows, we denote $c=c_{21}c_{12}<0$.   
        
        Now, we look for a symmetric positive definite matrix $W(t)$ that verifies the first and second-order symmetry equations. Considering $W(t)=(w_{ij}(t))$, the first equation is
        \begin{equation}           
         \label{edo1_diag} 
            \begin{pmatrix}
                2t w'_{11}+2(1-c_{11}+vt) w_{11}-2 c \Re(w_{12})    
                & 2t w'_{12}-c w_{22}-w_{11} \\2t \overline{w'_{12}}-c w_{22}-w_{11}
                & 2t w'_{22}+2(c_{11}-1-v t) w_{22}-2 \Re(w_{12})
                \end{pmatrix}=0.
                \end{equation}
and the second equation,
\begin{equation}           
         \label{edo2_diag} 
            \begin{pmatrix}
                t w''_{11}+(2-c_{11}+vt) w'_{11}+v w_{11}- c w'_{12}    
                & t w''_{12}+ (c_{11} -  v t) w'_{12} - w'_{11} \\
                t \overline{w''_{12}}+ (2-c_{11} +  v t) \overline{w'_{12}} -c w'_{22}
                & t w''_{22}+(c_{11}-vt)w'_{22}-v w_{22}-\overline{w'_{12}}
            \end{pmatrix}=0.
                \end{equation}

        The entry $(1,2)$ of \eqref{edo1_diag} leads to the relation
        \begin{equation}
            \label{w12d_diag}        
        w'_{12}(t)=\dfrac{w_{11}(t)+c w_{22}(t)}{2t}. 
        \end{equation}
        Substituting this relation into the entry $(1,1)$ of the (\ref{edo2_diag}) we obtain
                 \begin{equation}
            \label{w22_diag}
            w_{22}=\frac{2t^2 w''_{11}+2t(2-c_{11}+vt) w'_{11}+(2vt-c)w_{11}}{c^2}.
        \end{equation}
        
        Subsequently, the entry $(1,2)$ of (\ref{edo2_diag}) is a third-order differential equation for $w_{11}$, whose integrable solution to this equation is 
        \begin{equation*}
            w_{11}(t)=e^{-v t} t^{\sqrt{(1 - c_{11})^{2} +c}} k_1  U\left(\frac{1 - c_{11} +\sqrt{(1 - c_{11})^{2} +c}}{2}, 
            1 + \sqrt{(1 - c_{11})^{2} +c}, 
           v t\right)^2 ,
        \end{equation*}
         with  $(1 - c_{11})^{2} +c<1$ and $k_1\in\mathbb{C}$. Here $U$ denotes the confluent hypergeometric function of the second kind.
         However, the function $w_{22}$ given by (\ref{w22_diag}) is not integrable under this condition. With this, the demonstration concludes.
          
    \end{proof}


\section{Conclusions}\label{sec7}

    In the preceding sections, four potential configurations for $D$ and $W$ were derived from Theorems \ref{Teovequ}, \ref{Teoveq2u}, \ref{Teoueq2v} and \ref{Teovequesc}. Here, we present these theorems in equivalent forms and demonstrate the equivalence of two of them.
    \begin{theorem}\label{MainTheo}
        The second-order differential operator
             \begin{equation*}
                D = t \partial^2 + (C - t U ) \partial - V, \quad t\in(0,\infty), C,U,V\in \mathbb{C}^{2\times 2},
            \end{equation*}
        is symmetric with respect to some weight $W$ whose monic polynomials are irreducible if, and only if, the pair $(D, W)$ is equivalent to one of the following
        
        \begin{enumerate}
            \item $\begin{aligned}[t]
      D_{\alpha,\beta,b}&=t I\partial^2 +\begin{pmatrix}
                \alpha+\beta+1-t & 0\\
                -b(\beta-2)t & \alpha+1-t
             \end{pmatrix}\partial -\begin{pmatrix}
                 1 & 0\\
                 -b (\alpha+1) & 0
             \end{pmatrix},\\
      W_{\alpha, \beta,b}(t)&=e^{-t}t^\alpha \begin{pmatrix}
                t^\beta+b^2t^2
                & b t\\[1em]
                 b t
                 & 1
            \end{pmatrix},
            \quad  \alpha>-1, \quad \beta>-1-\alpha, \quad b\in \mathbb{R}-\{ 0 \};
      \end{aligned}$

      \item $\begin{aligned}[t]
           D_{\alpha,b}&=t I\partial^2 +\begin{pmatrix}
                \alpha+5-t & 0\\
                -4b(\alpha+2)t & \alpha+1-t
             \end{pmatrix}\partial -\begin{pmatrix}
                 2 & 0\\
                 -2b (\alpha+2)(\alpha+1) & 0
             \end{pmatrix},\\
            W_{\alpha, b}(t)&=e^{-t}t^{\alpha} \begin{pmatrix}
                t^4+4b^2(\alpha+2)t^2(\alpha+2-t)
                & b t(2(\alpha+2)-t) \\[1em]
                b t(2(\alpha+2)-t)
                & 1
            \end{pmatrix},\quad \alpha>-1,\quad  b\in\mathbb{R}, \, 0<|b|<1;
      \end{aligned}$
      
      \item $\begin{aligned}[t]
          D_{\beta}&=t I\partial^2 +\begin{pmatrix}
                3/2-t & \beta/4\\
                t & 1/2-t
             \end{pmatrix}\partial -\begin{pmatrix}
                 1/2 & 0\\
                 -1/2 & 0
             \end{pmatrix},\\
            W_{\beta}(t)&=e^{-t}t^{-1/2} \begin{pmatrix}
                \dfrac{4t(e^{\sqrt{\beta t}}+e^{-\sqrt{\beta t}})}{\beta}
                &  \dfrac{2\sqrt{t}(e^{\sqrt{\beta t}}-e^{-\sqrt{\beta t}})}{\sqrt{\beta}} \\[1em]
                \dfrac{2\sqrt{t}(e^{\sqrt{\beta t}}-e^{-\sqrt{\beta t}})}{\sqrt{\beta}}
                & e^{\sqrt{\beta t}}+e^{-\sqrt{\beta t}}
            \end{pmatrix}, \quad \beta>0.
      \end{aligned}$
        \end{enumerate}
        \end{theorem}
    
    \begin{proof}
        \begin{enumerate}
            \item Suppose $\beta\neq 2$. Take the matrix $M=\begin{pmatrix}
                \frac{u^{\beta/2}}{\sqrt{\gamma}}& 0\\
                \frac{u^{\frac{\beta-2}{2}}(\alpha+1)}{(\beta-2)\sqrt{\gamma}} & 1
            \end{pmatrix}$, where $u>0$. Then $W_{\alpha,\beta,b}=u^{\alpha}M^*WM$ and $D_{\alpha,\beta,b}=u^{-1}M^{-1}DM$, where $(W,D)$ is the pair given in Theorem \ref{Teovequ} with $c_{22}=\alpha+1$, $c_{21}=\frac{\beta(\alpha+1)}{u(\beta-2)}$, $b=\frac{u^{\beta/2}}{\sqrt{\gamma}(\beta-2)}$, and the change of variable $t=xu$.
        
            \
        
            Now suppose $\beta=2$ and take $M=\begin{pmatrix}
                    \frac{ub}{m\sqrt{\gamma}} & 0\\
                    b(\alpha+1) & 1
                \end{pmatrix}$, where $u>0$, $m=\frac{c_{21}u}{2(\alpha+1)\sqrt{\gamma}}$ and $b=\frac{|m|}{\sqrt{1-|m|^2}}$. Then, $W_{\alpha,b}(x)=u^{\alpha}M^*W M$ and $D_{\alpha,b}=u^{-1}M^{-1}DM$, where $(W,D)$ is the pair given in Theorem \ref{Teovequesc} with $c_{22}=\alpha+1$ and the change of variable $t=xu$.
        
            \item Consider the matrix $M=\begin{pmatrix}
                4ub(\alpha+2) & 0\\
                (\alpha+1)(\alpha+2) & 1
            \end{pmatrix}$, where $b=\frac{u}{4(\alpha+2)\sqrt{\gamma}}$. Then $W_{\alpha,b}(x)=u^{\alpha}M^*W M$ and $D_{\alpha,b}=u^{-1}M^{-1}DM$, where $(W,D)$ is the pair given in Theorem \ref{Teoveq2u} with $c_{21}=\frac{\alpha+1}{u}$ and the change of variable $t=xu$.
        
            \item Consider the matrix $M=\begin{pmatrix}
                -u & 0\\
                1 & 1
            \end{pmatrix}$. Then $W_{\beta}(x)=u^{-1/2}M^*W M$ and $D_{\beta}=u^{-1}M^{-1}DM$, where $(W,D)$ is the pair given in Theorem \ref{Teoueq2v} with $c_{22}=\frac{\beta+2}{4}$ and the change of variable $t=xu$.
        \end{enumerate}
    \end{proof}
\section{Aknowledgements}
The authors would like to thank Ines Pacharoni and Ignacio Zurri\'an for their careful review and the suggestions provided in this work.

\bibliographystyle{acm}

\end{document}